%% file: root.tex
\documentclass[9pt,shortpaper,twoside,web]{ieeecolor}

\usepackage[noadjust]{cite}
\usepackage{mathdots}
\usepackage{dsfont}
\usepackage{pgfplotstable}
\usepackage{pgfplots}
\let\labelindent\relax
\usepackage{enumitem}
\usepackage{generic}
\usepackage{comment}
\usepackage{mathtools,amsmath}
\mathtoolsset{showonlyrefs}
\usepackage{tabularx} 
\usepackage{diagbox}
\usepackage{graphicx}
\usepackage{mathrsfs}
\usepackage{cite}
\usepackage{amssymb,amsfonts}
\usepackage{subcaption}

\usepackage{pgfplots,tikz}
\usepackage{pgfplotstable}
\usepgfplotslibrary{units}
\usetikzlibrary{shapes,arrows,patterns,positioning,spy,backgrounds}
\tikzstyle{block} = [draw, fill=gray!20, rectangle, 
    minimum height=2em, minimum width=4em]
\tikzstyle{sum} = [draw, fill=gray!20, circle, node distance=1.5cm]
\tikzstyle{input} = [coordinate]
\tikzstyle{output} = [coordinate]
\tikzstyle{pinstyle} = [pin edge={to-,thin,black}]

\DeclareMathOperator{\rank}{rank}
\DeclareMathOperator{\tr}{tr}

\DeclareMathOperator{\modu}{mod}
\DeclareMathOperator{\spec}{spec}

\newcommand{\set}[2]{\left\{#1 \mid #2\right\}}
\newcommand{\norm}[1]{\left\|#1\right\|}

\usepackage{float}
\usepackage{accents}
\usepackage{multicol}
\usepackage{multirow}

\usepackage{array,multirow}
\usepackage{hhline}
\usepackage{arydshln}

\newtheorem{theorem}{Theorem}

\newtheorem{lemma}[theorem]{Lemma}
\newtheorem{example}[theorem]{Example}

\newtheorem{proposition}[theorem]{Proposition}
\newtheorem{problem}{Problem}

\newtheorem{remark}[theorem]{Remark}

\newtheorem{definition}[theorem]{Definition}

\usepackage{algorithm}
\usepackage{algpseudocode}
\algrenewcommand\algorithmicrequire{\textbf{Given:}}
\algrenewcommand\algorithmicensure{\textbf{Return:}}

\usepackage{algpseudocode}
\usepackage{algorithm}

\usetikzlibrary{calc}

\usepackage{textcomp}
\def\BibTeX{{\rm B\kern-.05em{\sc i\kern-.025em b}\kern-.08em
    T\kern-.1667em\lower.7ex\hbox{E}\kern-.125emX}}
\begin{document}
\title{Robust Stabilization of Discrete-Time Linear Systems Requires Nonlinear Dynamic Feedback}
\author{Amir Shakouri, Marieke Heidema, 
Henk J. van Waarde
\thanks{Amir Shakouri and Marieke Heidema contributed equally to this work. The work of Marieke Heidema was supported by the CogniGron research center and the Ubbo Emmius Funds. The work of Henk van Waarde was supported by
the Dutch Research Council under the NWO Talent Programme Veni
Agreement (VI.Veni.22.335).}
\thanks{The authors are with the Bernoulli Institute for Mathematics, Computer Science and Artificial Intelligence, University of Groningen, The Netherlands (e-mail: a.shakouri@rug.nl; h.m.heidema@rug.nl; h.j.van.waarde@rug.nl). 
}
}

\maketitle
\thispagestyle{empty}
\pagestyle{empty}

\begin{abstract}
This paper studies the problem of robust stabilization of linear input-state systems in discrete time. We prove that for \emph{any} compact set of stabilizable systems, there exists a dynamic state-feedback controller that globally asymptotically stabilizes all systems in the set. In addition, we show that for some compact sets of stabilizable systems, no nonlinear static or linear dynamic state-feedback law can achieve this task. This proves that, in general, robust stabilization \emph{requires} a feedback law that is both nonlinear and dynamic. We extend our study to robust exponential stabilization with a given rate of decay. Finally, for polytopic sets of systems, we introduce an algorithm for the design of robust feedback laws.
\end{abstract}


\input{introduction}

\input{preliminaries}

\input{problem_statement}

\input{quadratic_stabilization}

\input{main_results}

\input{fundamental_limitations}

\input{design}

\input{examples}

\input{conclusion}
\input{appendix}

\section*{References}

\bibliographystyle{IEEEtran}
\bibliography{biblo}

\end{document}

%% file: introduction.tex
\section{Introduction}
\label{sec:introduction}

The problem of robust stabilization deals with the design of a single feedback controller that simultaneously stabilizes all systems within a given family. This problem has a rich history, which especially took flight during the 1980s, leading to the celebrated solution to the $H_\infty$--optimal control problem \cite{doyle2002state}. The considered family of systems is typically formalized by a nominal system and an uncertainty description that characterizes how the true system may differ from the nominal one. Many different types of uncertainty descriptions have been studied, ranging from parametric to dynamic uncertainty, see, e.g., \cite{bhattacharyya1995robust,zhou1996robust,scherer1999lecture}.

An important preliminary problem to robust stabilization is that of robust stability, referring to the asymptotic stability of an entire family of autonomous systems. For some families of linear time-invariant (LTI) systems, robust stability analysis can be accomplished using Kharitonov's theorem \cite{kharitonov1978asymptotic} or its extensions \cite{barmish1989generalization,anderson2003robust,hollot2003some}. However, it is known that in general, there is no numerically tractable method to assess robust stability of an arbitrary family of systems \cite{blondel1997np}. This has motivated the study of \emph{quadratic stability}. A family of systems is called quadratically stable if all its members admit a \emph{common} quadratic Lyapunov function \cite{petersen1987notions}. For several families of systems, quadratic stability can be assessed numerically using, for example, linear matrix inequalities (LMIs) \cite{khargonekar2002robust,amato2002note}. However, not all robustly stable families of systems admit a common Lyapunov function, meaning that quadratic stability is, in general, only a sufficient condition for robust stability. To reduce the conservativeness of quadratic stability, several works instead seek a \emph{parameter-dependent} Lyapunov function \cite{feron2002analysis}.

For open dynamical systems, a well-studied concept is that of \emph{quadratic stabilization}. A family of input-state systems is called \emph{quadratically stabilizable} if there exists a single linear static state-feedback that stabilizes all systems in the family with a common quadratic Lyapunov function \cite{barmish1985necessary,petersen1987stabilization}. Analogous to the case of quadratic stability, quadratic stabilization is, in general, a conservative sufficient condition for robust stabilization. This is not only due to the fact that a common Lyapunov function may not exist, but also due to the requirement of a \emph{linear static} feedback controller. As noted in \cite{petersen2003quadratic}, for some families of systems that are not quadratically stabilizable, there exist nonlinear static feedback laws that stabilize all systems within the family. Beyond static feedback laws, it has been shown, e.g., in \cite{fu1986adaptive} that, for \emph{any} compact family of stabilizable systems, there exists a switching controller that stabilizes all systems within the family. However, despite these relevant findings, there are gaps in the existing literature. In particular, it is unclear what the fundamental limitations of nonlinear static feedback laws are. Moreover, the potential of nonlinear \emph{dynamic} feedback laws has not yet been fully explored. 

In this paper, we focus on robust stabilization of discrete-time linear input-state systems. We study general uncertainty descriptions in the form of arbitrary compact sets of stabilizable systems. Our main contribution is threefold:
\begin{enumerate}[left=0pt, labelsep=0.5em]
    \item We show that for every compact set of stabilizable systems there exists a (time-invariant) \emph{nonlinear dynamic} state-feedback that globally stabilizes all members of the set (see Theorem~\ref{th:dynamic}). We also provide a method to construct such feedback laws (see Theorem~\ref{th:dynamic_cont} and Algorithm~\ref{alg:1}).
    \item We study robust \emph{exponential} stabilization. We show that whenever a robust stabilizing feedback law exists, there also exists one that achieves exponential stability (see Theorem~\ref{th:dynamic}). Moreover, we extend the framework to global exponential stabilization with a given rate of decay. 
    \item We study fundamental limitations of both nonlinear static and linear dynamic feedback laws. We show that for some compact sets of systems, no such feedback laws can simultaneously stabilize all members of the set (see Theorems~\ref{thm: static state-feedback} and~\ref{thm: dynamic linear}). 
\end{enumerate}

The remainder of this paper is organized as follows. In Section~\ref{sec:II}, we discuss notation and preliminaries. Section~\ref{sec:III} then formalizes the research problem. In Section~\ref{sec:quadratic stabilizability}, we introduce and study the notion of $\alpha$--quadratic stabilizability. 
Section~\ref{sec:IV} studies robust stabilization using nonlinear dynamic state-feedback, providing the solution to the research problem. Fundamental limitations of nonlinear static and linear dynamic feedback laws are discussed in Section~\ref{sec:V}. 
Section~\ref{sec:VI} provides algorithms for covering a polytopic set of systems by a finite number of subsets which are, individually, $\alpha$--quadratic stabilizable. 
The results are then applied to an example in Section~\ref{sec: examples} and concluded upon in Section~\ref{sec: conclusion}.

%% file: Preliminaries.tex
\section{Preliminaries} \label{sec:II}

\subsection{Notation}

Let $\mathbb{Z}_+$ and $\mathbb{N}$ denote the sets of nonnegative and positive integers, respectively. For vectors $x,y\in\mathbb{R}^n$, the inequality $x\preceq y$ is meant element-wise, i.e., $x_i\leq y_i$ for all \mbox{$i\in\{1,\ldots,n\}$}. The spectral norm of a matrix $M\in\mathbb{R}^{n\times m}$ is denoted by $\|M\|$. The ball of matrix pairs centered around $(\hat{A},\hat{B})\in\mathbb{R}^{n\times n}\times\mathbb{R}^{n\times m}$ with radius $r\geq 0$ is denoted by
\begin{equation}
\mathcal{B}_r(\hat{A},\hat{B})\!\coloneqq\! \set{(A,B)\!\in\!\mathbb{R}^{n\times n}\times \mathbb{R}^{n\times m}}{\norm{\begin{bmatrix}
    \hat{A}\!-\!A \!\!&\!\! \hat{B}\!-\!B
\end{bmatrix}}\!\leq\! r}.
\end{equation}
For a square matrix $M$, we denote the set of its eigenvalues by $\spec(M)$. We say $M$ is positive definite (resp., positive semi-definite), denoted by $M > 0$ (resp., $M\geq 0$), if it is symmetric and every $\lambda\in\spec(M)$ satisfies $\lambda>0$ (resp., $\lambda\geq 0$). 


\subsection{Partial Lyapunov stability}

Partial stability extends the usual notion of Lyapunov stability to one that only concerns the behavior of a subset of the state variables (see \cite{vorotnikov1998partial}). Let $n_1,n_2\in\mathbb{N}$ and consider the discrete-time system
\begin{equation}
\label{eq:sys0}
\xi(t+1)=h(\xi(t)),
\end{equation}
where $t\in\mathbb{Z}_+$, $\xi(t)\in\mathbb{R}^{n_1+n_2}$, and $h:\mathbb{R}^{n_1+n_2}\rightarrow \mathbb{R}^{n_1+n_2}$. We partition $\xi(t)$ as $\xi(t)^\top=\begin{bmatrix}
\xi_1(t)^\top \!\!&\!\!
\xi_2(t)^\top
\end{bmatrix}$, where $\xi_1(t)\in\mathbb{R}^{n_1}$ and $\xi_2(t)\in\mathbb{R}^{n_2}$.
Let $\xi_1(t,\xi_{10},\xi_{20})$ and $\xi_2(t,\xi_{10},\xi_{20})$, respectively, denote the states of \eqref{eq:sys0} corresponding to initial conditions \mbox{$\xi_1(0)=\xi_{10}\in\mathbb{R}^{n_1}$} and \mbox{$\xi_2(0)=\xi_{20}\in\mathbb{R}^{n_2}$}. We recall the definition of partial Lyapunov stability as follows. 

\begin{definition}[{\cite[Def. 13.6]{haddad2008nonlinear}}]
\label{def:partial_Lyap}
We say that system \eqref{eq:sys0} is
\begin{enumerate}[label=(\roman*),ref=\ref{def:partial_Lyap}(\roman*)]
    \item\label{def:partial_Lyap(i)} \emph{Lyapunov stable with respect to $\xi_1$} if for every \mbox{$\varepsilon>0$} there exists $\delta>0$ such that if $\norm{\xi_{10}}\leq\delta$, then $\norm{\xi_1(t,\xi_{10},\xi_{20})}\leq\varepsilon$ for all $t\in\mathbb{Z}_+$ and all $\xi_{20}\in\mathbb{R}^{n_2}$. 
    \item\label{def:partial_Lyap(ii)} \emph{globally asymptotically stable with respect to $\xi_1$} if it is Lyapunov stable with respect to $\xi_1$ and for every $\xi_{10}\in\mathbb{R}^{n_1}$ and \mbox{$\xi_{20}\in\mathbb{R}^{n_2}$} we have \mbox{$\lim_{t\rightarrow\infty}\xi_1(t,\xi_{10},\xi_{20})=0$}. 
\end{enumerate}
\end{definition}

We note that in Definition~\ref{def:partial_Lyap(i)}, the bound $\delta$ does not depend on the second part of the initial state $\xi_{20}$. This makes both notions of Definition~\ref{def:partial_Lyap} to be uniform\footnote{Because of this property, the notion in Definition~\ref{def:partial_Lyap} is also known as Lyapunov stability with respect to $\xi_1$ uniformly in $\xi_2$, see \cite[Def. 13.6]{haddad2008nonlinear}. For the sake of simplicity, we drop the emphasis on uniformity in this paper.} in $\xi_{20}$.

%% file: problem_statement.tex
\section{Problem Formulation}
\label{sec:III}

Let $n,m\in\mathbb{N}$, and consider the discrete-time LTI system
\begin{equation}
\label{eq:1}
x(t+1)=Ax(t)+Bu(t),
\end{equation}
where $t\in\mathbb{Z}_+$ denotes time, $x(t)\in\mathbb{R}^n$ is the state, and $u(t)\in\mathbb{R}^m$ is the input.
We identify the system~\eqref{eq:1} with the pair of matrices $(A,B)\in\Sigma\coloneqq \mathbb{R}^{n\times n}\times \mathbb{R}^{n\times m}$.

Consider the system $(A_\text{true},B_\text{true})\in\Sigma$, referred to as the true system. The true system is considered unknown, but contained in a given \emph{compact} set $\Sigma_{\text{pk}}\subset\Sigma$, i.e.,
$(A_\text{true},B_\text{true})\in\Sigma_{\text{pk}}$.
The set $\Sigma_{\text{pk}}$ represents our \textit{prior knowledge} of the true system\footnote{For instance, one may consider $\Sigma_{\text{pk}}=\{(A_0,B_0)\}+\Delta$, where $(A_0,B_0)$ is a nominal system and $\Delta\subset\Sigma$ is an uncertainty set.}. The problem of robust stabilization is to find a \emph{single} feedback law that simultaneously stabilizes \emph{all} systems within $\Sigma_{\text{pk}}$. This way, one guarantees that the unknown true system is also stabilized. 

In this paper, we consider \emph{dynamic} state-feedback laws of the form
\begin{equation}
\label{eq:fd}
\begin{split}
    z(t+1)&=f(x(t),z(t)), \\ 
    u(t)&=g(x(t),z(t)),
\end{split}
\end{equation} 
where $z\in\mathbb{R}^p$ is the state of the controller, and \mbox{$f:\mathbb{R}^n\times\mathbb{R}^p\rightarrow \mathbb{R}^p$} and $g:\mathbb{R}^n\times\mathbb{R}^p\rightarrow \mathbb{R}^m$ are functions of $x(t)$ and $z(t)$. The closed-loop system obtained by interconnecting the system \eqref{eq:1} and the feedback law \eqref{eq:fd} is
\begin{equation}
\label{eq:def2}
\begin{split}
x(t+1)&=Ax(t)+Bg(x(t),z(t)), \\
z(t+1)&=f(x(t),z(t)).
\end{split}
\end{equation}
By $x(t,x_0,z_0)$ and $z(t,x_0,z_0)$, respectively, we denote the solutions $x(t)$ and $z(t)$ generated by \eqref{eq:def2} starting from the initial conditions $x(0)=x_0$ and $z(0)=z_0$.



Now, our objective is to find a single feedback law~\eqref{eq:fd} that stabilizes~\eqref{eq:1} for all $(A,B)\in\Sigma_\text{pk}$. Depending on the given $\Sigma_\text{pk}$, this may be possible or not, motivating the following definition. 

\begin{definition} 
\label{def:robust dynamic stab}
We say that $\Sigma_{\text{pk}}$ \emph{enables robust stabilization} if there exist $p\in\mathbb{N}$, \mbox{$f:\mathbb{R}^n\times\mathbb{R}^p\rightarrow \mathbb{R}^p$}, and $g:\mathbb{R}^n\times\mathbb{R}^p\rightarrow \mathbb{R}^m$ such that the following properties are satisfied:
\begin{enumerate}[label=(P\arabic*),ref=(P\arabic*)]
    \item\label{(P2)} For every $(A,B)\in\Sigma_{\text{pk}}$ the closed-loop system \eqref{eq:def2} is globally asymptotically stable with respect to $x$. 
    \item\label{(P1)} For every $x_0\in\mathbb{R}^n$, $z_0\in\mathbb{R}^p$, and $(A,B)\in\Sigma_{\text{pk}}$, $z(t,x_0,z_0)$ is bounded. 
\end{enumerate}
\end{definition}

Note that \ref{(P2)} ensures that the system's state $x(t)$ is Lyapunov stable and converges to zero, while \ref{(P1)} means that the controller's state remains bounded. The latter is motivated by the fact that in the absence of \ref{(P1)}, the computation of $z(t)$, and thus $u(t)$, would suffer from numerical problems. 


In practice, it is often important to ensure that the closed-loop system converges \emph{exponentially} with a certain rate. This motivates the following definition. 

\begin{definition} 
\label{def:alph-robust exp stab}
Let $\alpha\in(0,1]$. We say that $\Sigma_{\text{pk}}$ \emph{enables robust \mbox{$\alpha$--exponential} stabilization} if there exist $p\in\mathbb{N}$, \mbox{$f:\mathbb{R}^n\times\mathbb{R}^p\rightarrow \mathbb{R}^p$}, and $g:\mathbb{R}^n\times\mathbb{R}^p\rightarrow \mathbb{R}^m$ such that \ref{(P1)} holds and the following property is satisfied:
\begin{enumerate}[label=(P\arabic*),ref=(P\arabic*)]\setcounter{enumi}{2}
    \item\label{(P4)} There exists $c\geq 1$ and $\eta\in[0,\alpha)$ such that for every \mbox{$(A,B)\in\Sigma_{\text{pk}}$}, $x_0\in\mathbb{R}^n$, and $z_0\in\mathbb{R}^{p}$, we have  $\norm{x(t,x_0,z_0)}\leq c \eta^t \norm{x_0}$ for all \mbox{$t\in\mathbb{Z}_+$}. 
\end{enumerate}
In case $\alpha=1$, we simply say that $\Sigma_{\text{pk}}$ \emph{enables robust exponential stabilization}. 
\end{definition}



In this work, we consider the following problem. 
\begin{problem}
\label{prob:1}
Suppose that $\Sigma_{\text{pk}}$ is compact. Let $\alpha\in(0,1]$. Provide necessary and sufficient conditions under which $\Sigma_{\text{pk}}$ enables robust (\mbox{$\alpha$--exponential}) stabilization. Moreover, if possible, find \mbox{$p\in\mathbb{N}$}, \mbox{$f:\mathbb{R}^n\times\mathbb{R}^p\rightarrow \mathbb{R}^p$}, and \mbox{$g:\mathbb{R}^n\times\mathbb{R}^p\rightarrow \mathbb{R}^m$} such that properties \ref{(P2)} and \ref{(P1)} (or \ref{(P1)} and \ref{(P4)}) are satisfied. 
\end{problem}



In our problem formulation, we focused on feedback laws that are, in general, dynamic and nonlinear. The reason for this will become clear in Section~\ref{sec:V}, where we discuss fundamental limitations of both \emph{nonlinear static} and \emph{linear dynamic} feedback laws. 

%% file: quadratic_stabilization.tex
\section{Quadratic Stabilization}
\label{sec:quadratic stabilizability}

A well-known approach towards finding a robust \emph{linear static} feedback is so-called quadratic stabilization. In this section, we revisit this approach by introducing some notions and providing a few new results. These results are instrumental in solving Problem~\ref{prob:1} later on. 

We start from the following definition of $\alpha$--stabilizability.

\begin{definition}
Let $(A,B)\in\Sigma$ and $\alpha\in(0,1]$. We say that $(A,B)$ is $\alpha$--stabilizable if there exists $K\in\mathbb{R}^{m\times n}$ such that every \mbox{$\lambda\in\spec(A+BK)$} satisfies $|\lambda|<\alpha$. 
\end{definition}

Note that $1$--stabilizability coincides with stabilizability. The following lemma provides conditions for $\alpha$--stabilizability in terms of a Hautus-like test and a Lyapunov inequality. The proof of this lemma can be found in Appendix~\ref{app:proof of alpha_Hautus}.


\begin{lemma}
\label{lem:alpha_Hautus}
Let $(A,B)\in\Sigma$ and $\alpha\in(0,1]$. Then, the following statements are equivalent:
\begin{enumerate}[label=(\alph*),ref=\ref{lem:alpha_Hautus}(\alph*)]
    \item\label{lem:alpha_Hautus(a)} $(A,B)$ is $\alpha$--stabilizable.
    \item\label{lem:alpha_Hautus(b)} $\rank\begin{bmatrix}A-\lambda I & B\end{bmatrix}=n$ for all $\lambda\in\mathbb{C}$ with $|\lambda|\geq \alpha$.
    \item\label{lem:alpha_Hautus(c)} There exists $P>0$ and $K$ such that
    \begin{equation}
    \label{eq:lyap}
    \alpha^2 P-(A+BK)^\top P(A+BK) >0.
    \end{equation}
\end{enumerate}
\end{lemma}\vspace{0.25 cm}

For $\alpha\in(0,1]$, we denote the set of $\alpha$--stabilizable systems by
\begin{equation}
\Sigma_\text{stab}(\alpha)\coloneqq \set{(A,B)\in\Sigma}{(A,B)\text{ is }\alpha\text{--stabilizable}}.
\end{equation}
In particular, we denote the set of stabilizable systems by \mbox{$\Sigma_\text{stab}=\Sigma_\text{stab}(1)$}. We note that for all $\alpha_1,\alpha_2\in(0,1]$ such that $\alpha_1\leq \alpha_2$, we have $\Sigma_\text{stab}(\alpha_1) \subseteq \Sigma_\text{stab}(\alpha_2)$. 

We now turn our attention to the notion of $\alpha$--quadratic stabilizability that deals with a set of systems.

\begin{definition}
\label{def:QS}
Let $\alpha\in(0,1]$ and $\mathcal{S}\subseteq\Sigma_\text{stab}(\alpha)$. We say that $\mathcal{S}$ is $\alpha$--\emph{quadratically stabilizable} if there exist $P>0$ and $K$ such that \eqref{eq:lyap} holds for all $(A,B)\in\mathcal{S}$. In case $\alpha=1$, we simply say that $\mathcal{S}$ is \emph{quadratically stabilizable}. 
\end{definition}

For \emph{compact} sets of systems, $\alpha$--quadratic stabilizability can be characterized using a nonstrict inequality instead of the strict one in~\eqref{eq:lyap}. This is discussed in the following lemma, which will be used further on in the paper. The proof of this lemma can be found in Appendix~\ref{app:proof of QS_compact}. 

\begin{lemma}
\label{lem:QS_compact}
Let $\alpha\in(0,1]$ and $\mathcal{S}\subseteq\Sigma_\text{stab}(\alpha)$ be compact. Then, $\mathcal{S}$ is $\alpha$--quadratically stabilizable \emph{if and only if} there exists $P\geq I$ and $K$ such that
\begin{equation}
\label{eq:Lyap_P^2}
\alpha^2 P-(A+BK)^\top P(A+BK)\geq I
\end{equation}
for all $(A,B)\in\mathcal{S}$.
\end{lemma}

Now, for a pair $(\hat{A},\hat{B})\in\Sigma_\text{stab}(\alpha)$, we define its \emph{radius of \mbox{$\alpha$--quadratic} stabilizability} as
\begin{equation}
\label{eq:rad_sup}
\rho_\alpha(\hat{A},\hat{B})\coloneqq\sup\set{r}{\mathcal{B}_r(\hat{A},\hat{B})\text{ is }\alpha \text{--quadratically stabilizable}}.
\end{equation}
In case $(\hat{A},\hat{B})$ is not $\alpha$--stabilizable, we define $\rho_\alpha(\hat{A},\hat{B})=0$.

We elaborate on this and other properties of the radius in the following lemma. The proof of this lemma can be found in Appendix~\ref{app: proof of lem sdp}. 
\begin{lemma}
\label{lem:properties}
Let $(\hat{A},\hat{B})\in\Sigma$ and $\alpha\in(0,1]$. Then: 
\begin{enumerate}[label=(\alph*),ref=\ref{lem:properties}(\alph*)]
    \item\label{lem:properties(a)} $0\leq\rho_\alpha(\hat{A},\hat{B})\leq \alpha$.
    \item\label{lem:properties(b)} $\rho_\alpha(\hat{A},\hat{B})=\alpha$ \emph{if and only if} $\hat{A}=0$.
    \item\label{lem:properties(bc)} $\rho_\alpha(\hat{A},\hat{B})=0$ \emph{if and only if} $(\hat{A},\hat{B})$ is not $\alpha$--stabilizable.
\end{enumerate} 
\end{lemma}

We note that, in view of Lemma~\ref{lem:properties(a)}, not all sets of \mbox{$\alpha$--stabilizable} systems are $\alpha$--quadratically stabilizable. As an example, consider $\Sigma_\text{pk}\subset\Sigma_\text{stab}(\alpha)$ large enough so that it contains a ball of systems with radius larger than or equal to $\alpha$. In this case, robust stabilization through finding an \mbox{$\alpha$--quadratically} stabilizing feedback gain is not feasible due to Lemma~\ref{lem:properties(a)}. As a result, for such cases, one has to go beyond the framework of quadratic stabilization by linear feedback. We return to this point in Section~\ref{sec:IV-static}, where we elaborate on the limitations of static feedback laws in detail. Before that, in Section~\ref{sec:IV}, we use the results of this section to construct robust feedback laws that are dynamic and possibly nonlinear.

%% file: main_results.tex
\section{Nonlinear Dynamic State-Feedback}
\label{sec:IV}

In this section, we study robust stabilization using dynamic state-feedback and provide the solution to Problem~\ref{prob:1}. 
We start with the following theorem showing that for \emph{any} compact set $\Sigma_\text{pk}\subset\Sigma_\text{stab}(\alpha)$ there exists a (possibly nonlinear) dynamic state-feedback law that simultaneously stabilizes all systems in~$\Sigma_\text{pk}$ with convergence \mbox{rate $\alpha$}. 

\begin{theorem}
\label{th:dynamic}
Suppose that $\Sigma_\text{pk}$ is compact. Let $\alpha\in(0,1]$. Then, $\Sigma_\text{pk}$ enables robust $\alpha$--exponential stabilization \emph{if and only if} \mbox{$\Sigma_\text{pk}\subset\Sigma_\text{stab}(\alpha)$}. Moreover, the following statements are \emph{equivalent}:
\begin{enumerate}[label=(\alph*),ref=\ref{th:dynamic}(\alph*)]
    \item\label{th:dynamic(a)} $\Sigma_\text{pk}$ enables robust stabilization.
    \item\label{th:dynamic(b)} $\Sigma_\text{pk}$ enables robust exponential stabilization.
    \item\label{th:dynamic(c)} $\Sigma_\text{pk}\subset\Sigma_\text{stab}$.
\end{enumerate}
\end{theorem}

We postpone the proof of this theorem until after Theorem~\ref{th:dynamic_cont}, presented later on in this section, where it is shown how a stabilizing dynamic state-feedback can be constructed. 
In our construction of a robust feedback law, we leverage the following result, which shows that every compact set of $\alpha$--stabilizable systems can be covered by a finite number of \mbox{$\alpha$--quadratically} stabilizable subsets (cf. \cite[Lem. 3.1]{fu1986adaptive}). Constructing such subsets using an algorithm is discussed later on in Section~\ref{sec:VI}.

\begin{lemma}
\label{lem:finite covering quadratic stab sets}
Let $\alpha\in(0,1]$. Suppose that $\Sigma_\text{pk}\subset\Sigma_\text{stab}(\alpha)$ is compact. Then, there exist $q\in\mathbb{N}$ and matrices $K_i\in\mathbb{R}^{m\times n}$ and $P_i\geq I$ for $i\in\{1,\ldots,q\}$, such that every $(A,B)\in\Sigma_\text{pk}$ satisfies
\begin{equation}
\label{eq:lyap_i}
\alpha^2 P_i-(A+BK_i)^\top P_i(A+BK_i)\geq I
\end{equation}
for some $i\in\{1,\ldots,q\}$. 
\end{lemma}
\begin{proof}
We define $\rho_\alpha(\Sigma_\text{pk})\coloneqq\inf\set{\rho_\alpha(A,B)}{(A,B)\in\Sigma_\text{pk}}$. It follows from Lemma~\ref{lem:properties(bc)} and the compactness of $\Sigma_\text{pk}$ that $\rho_\alpha(\Sigma_\text{pk})>0$. Let $r\in(0,\rho_\alpha(\Sigma_\text{pk}))$. Since $\Sigma_\text{pk}$ is bounded, it can be covered by finitely many balls of radius $r$. Let $q\in\mathbb{N}$ and \mbox{$(A_i,B_i)\in\Sigma_\text{pk}$}, for $i\in\{1,\ldots,q\}$, be such that \mbox{$\Sigma_\text{pk}\subseteq\bigcup_{i=1}^q\mathcal{B}_{r}(A_i,B_i)$}. Now, take \mbox{$\Sigma_i=\mathcal{B}_{r}(A_i,B_i)$}. Since $(A_i,B_i)\in\Sigma_\text{pk}$, we have $r<\rho_\alpha(A_i,B_i)$ for all $i\in\{1,\ldots,q\}$. Therefore, the set $\Sigma_i$, $i\in\{1,\ldots,q\}$, is \mbox{$\alpha$--quadratically} stabilizable. Now, it follows from Lemma~\ref{lem:QS_compact} that for every $i\in\{1,\ldots,q\}$, there exists \mbox{$P_i>0$} and $K_i$ such that \mbox{$\alpha^2 P_i-(A+BK_i)^\top P_i(A+BK_i)\geq I$} holds for all $(A,B)\in\Sigma_i$. Since $\Sigma_\text{pk}\subseteq\bigcup_{i=1}^q\Sigma_i$, the matrices $P_i$ and $K_i$ have the property that for every $(A,B)\in\Sigma_\text{pk}$ there exists \mbox{$i\in\{1,\ldots,q\}$} such that \mbox{$\alpha^2P_i-(A+BK_i)^\top P_i(A+BK_i)\geq I$} holds.
\end{proof}



Now, we formulate a feedback law as follows. To simplify the notation, the explicit time dependence of $x(t)$ and $z(t)$ is omitted. Let \mbox{$\alpha\in(0,1]$}. We take $q\in\mathbb{N}$, $P_i>0$, and $K_i$, \mbox{$i\in\{1,\ldots,q\}$}, such that they satisfy the conditions in Lemma~\ref{lem:finite covering quadratic stab sets}. We take \mbox{$p=2$}, and we partition $z\in\mathbb{R}^{2}$ as $z =
\begin{bmatrix}
z_1 & z_2
\end{bmatrix}^\top$.
We also define the function \mbox{$\phi:\mathbb{R}^n\times \mathbb{R}^2\rightarrow\{1,\ldots,q\}$} as
\begin{equation} \phi(x,z)\coloneqq \left\{\begin{array}{ll}
z_2 & \text{if }z_2\in\{1,\ldots,q\}, x^{\top} P_{z_2} x\leq z_1, \\
\!(z_2 \modu q)+ 1 & \text{if }z_2\in\{1,\ldots,q\}, x^{\top} P_{z_2} x> z_1,\\
1 & \text{otherwise}.
\end{array}\right. 
\end{equation}
We take functions $f(x,z)$ and $g(x,z)$ as
\begin{subequations}
\label{eq:fun_fg}
\begin{align}
\label{eq:fun_f}
&f(x,z)=\begin{bmatrix}
\alpha^2 x^\top P_{\phi(x,z)} x-\|x\|^2 \\ \phi(x,z)
\end{bmatrix},\\
\label{eq:fun_g}
&g(x,z)=K_{\phi(x,z)}x.
\end{align}
\end{subequations}

Now, the following theorem shows that for any compact set of \mbox{$\alpha$--stabilizable} systems, a feedback law as in~\eqref{eq:fd}, with functions $f$ and $g$ of the form \eqref{eq:fun_f} and \eqref{eq:fun_g}, robustly stabilizes all systems in $\Sigma_\text{pk}$ with convergence rate $\alpha$. 

\begin{theorem}
\label{th:dynamic_cont}
Let $\alpha\in[0,1)$. Suppose that $\Sigma_\text{pk}\subset\Sigma_\text{stab}(\alpha)$ is compact. Let $P_i>0$ and $K_i$, $i\in\{1,\ldots,q\}$, satisfy the conditions of Lemma~\ref{lem:finite covering quadratic stab sets}. Let the functions $f$ and $g$ be as in \eqref{eq:fun_f} and \eqref{eq:fun_g}. Then, properties \ref{(P1)} and \ref{(P4)} hold.
\end{theorem}

The proof of Theorem~\ref{th:dynamic_cont} can be found in Appendix~\ref{app: proof of th dynamic_cont}.
The proof of Theorem~\ref{th:dynamic} now follows from Theorem~\ref{th:dynamic_cont}. 

\textit{Proof of Theorem~\ref{th:dynamic}:} 
Let us first show that $\Sigma_\text{pk}$ enables robust $\alpha$--exponential stabilization if and only if \mbox{$\Sigma_\text{pk}\subset\Sigma_\text{stab}(\alpha)$}. 
The ``if'' part follows immediately from Theorem~\ref{th:dynamic_cont}. For the ``only if'' part, assume that there exists $(A,B)\in\Sigma_\text{pk}$ that is not \mbox{$\alpha$--stabilizable}. Let $f$ and $g$ be arbitrary. Let $v\in\mathbb{C}^n$ and $\lambda\in\mathbb{C}$ be such that $\|v\|=1$, \mbox{$v^* B=0$}, $v^* A=\lambda v^*$, and $|\lambda|\geq\alpha$. Take $x_0\in\mathbb{R}^n$ such that \mbox{$|v^*x_0|=1$}. Observe that $v^* x(t,x_0,z_0)=\lambda^t v^*x_0$ for all $z_0\in\mathbb{R}^p$ and $t\in\mathbb{Z}_+$. Thus, we have \mbox{$|v^* x(t,x_0,z_0)|=|\lambda|^t \geq \alpha^t$}.
Since \mbox{$\|v\|\|x(t,x_0,z_0)\|\geq|v^* x(t,x_0,z_0)|$}, we have \mbox{$\|x(t,x_0,z_0)\|\geq \alpha^t$}.
This implies that there are no $c\geq1$ and $\eta\in[0,\alpha)$ such that $\|x(t,x_0,z_0)\|\leq c \eta^t$ holds for all $t\in\mathbb{Z}_+$. Therefore, there are no $f$ and $g$ satisfying property~\ref{(P4)}, hence, $\Sigma_\text{pk}$ does not enable \mbox{$\alpha$--exponential} stabilization. 

Let us now show that (a), (b), and (c) are equivalent.
The equivalence between (b) and (c) follows from the first statement of Theorem~\ref{th:dynamic} by setting $\alpha=1$. In addition, it is evident from Definitions~\ref{def:robust dynamic stab} and~\ref{def:alph-robust exp stab} that (b) implies (a). It thus suffices to show that (a) implies (c). For this, suppose that (c) does not hold. Hence, there exists $(A,B)\in\Sigma_\text{pk}$ that is not stabilizable. Let $f$ and $g$ be arbitrary. Let $v\in\mathbb{C}^n$ and $\lambda\in\mathbb{C}$ be such that $\|v\|=1$, $v^* B=0$, $v^* A=\lambda v^*$, and $|\lambda|\geq1$. Take $x_0\in\mathbb{R}^n$ such that $|v^*x_0|=1$. One can follow the same steps as before, with $\alpha=1$, to show that the state trajectory of the closed-loop system satisfies $\|x(t,x_0,z_0)\|\geq 1$ for all $t\in\mathbb{Z}_+$. Therefore, there are no $f$ and $g$ satisfying property~\ref{(P2)}, and thus, (a) does not hold. \hfill\QED

Problem~\ref{prob:1} has now been solved in Theorems~\ref{th:dynamic} and~\ref{th:dynamic_cont}. The proposed feedback law in \eqref{eq:fun_fg} behaves as follows. At each $t\in\mathbb{Z}_+$, the input takes the form \mbox{$u(t)=K_ix(t)$} for some $i\in\{1,\ldots,q\}$. In case the system's trajectory satisfies the Lyapunov inequality \eqref{eq:Lyap_P^2} with $K=K_i$ and $P=P_i$, we keep using the same feedback gain $K_i$. In case the inequality is violated, we move to another feedback gain at the next time step, namely, \mbox{$u(t+1)=K_{i+1}x(t+1)$} if $i<q$ and $u(t+1)=K_{1}x(t+1)$ if $i=q$. 
 
 \begin{remark}
The feedback law \eqref{eq:fun_fg} has a switching nature, where the switching rule is only \mbox{state-dependent} and embedded in the dynamics of the controller's state $z(t)$. This has similarities to the idea of adaptive switching control in \cite{fu1986adaptive}, which deals with continuous-time systems. There, the switching controller is also based on a covering of $\Sigma_\text{pk}$ by a finite number of sets, each of which admits a robust linear feedback law. The controller in \cite{fu1986adaptive} switches between such feedback laws due to a switching rule that is both time and state-dependent. Another line of research that shares a similar idea is that of supervisory control \cite{morse1996supervisory} and adaptive control using multiple models \cite{narendra1997adaptive}. However, both of these methods need $\Sigma_\text{pk}$ to be finite.
 \end{remark}

\begin{remark}
    The dimension $p\in\mathbb{N}$ of the controller's state is a measure for the complexity of the controller. It is desirable to find controllers with the \emph{least} possible number of states.
    Interestingly, the proposed feedback law has $p=2$ regardless of the system dimensions $n$ and $m$. 
    Without going into too much detail, we note that this can be further reduced to $p=1$ by encoding the two controller states in \eqref{eq:fun_fg} into one, e.g., using a sigmoid function. This is possible since the second entry of $z(t)$ in \eqref{eq:fun_fg} is an integer for all $t\geq 1$. 
\end{remark}

%% file: fundamental_limitations.tex
\section{Fundamental Limitations of Static and Linear Dynamic Feedback}
\label{sec:V}

In the previous section, it was shown by Theorem \ref{th:dynamic} that for any compact set $\Sigma_\text{pk}\subset \Sigma_\text{stab}$, there exists a (nonlinear) dynamic state-feedback law that simultaneously stabilizes all systems within $\Sigma_\text{pk}$. In this section, we show that for some compact sets $\Sigma_\text{pk}\subset \Sigma_\text{stab}$, there exist neither static nor linear dynamic state-feedback laws\footnote{We note that similar claims have been made for other notions of stability in input-output systems that are different from the asymptotic stability notion considered here. For instance, it was shown in \cite{khargonekar1987robust,cusumano1988nonlinear} that, for some families of systems, no linear or nonlinear controller can achieve a \emph{common} bound on the induced input-output gain over the entire family.} to stabilize all systems within $\Sigma_\text{pk}$. Therefore, in general, robust stabilization \textit{requires} a state-feedback that is both nonlinear and dynamic.

\subsection{Fundamental limitation of static state-feedback}
\label{sec:IV-static}
In this section, we consider (nonlinear) static state-feedback laws of the form
\begin{equation}
\label{eq:fd_static}
u(t)=g(x(t)),
\end{equation}
where $g:\mathbb{R}^n\rightarrow\mathbb{R}^m$. We investigate whether for any compact set $\Sigma_\text{pk}\subseteq \Sigma_\text{stab}$ there exists such a feedback that simultaneously stabilizes all systems within $\Sigma_\text{pk}$. The closed-loop system constructed by \eqref{eq:1} and \eqref{eq:fd_static} is of the form
\begin{equation}
\label{eq:def1}
x(t+1)=Ax(t)+Bg(x(t)).
\end{equation} 
Here, we consider a definition analogous to  Definition~\ref{def:robust dynamic stab} that concerns static feedback laws. 

\begin{definition}
We say that $\Sigma_\text{pk}$ \emph{enables robust stabilization by static state-feedback} if there exists $g:\mathbb{R}^n\rightarrow \mathbb{R}^m$ such that \eqref{eq:def1} is globally asymptotically stable for all $(A,B)\in\Sigma_\text{pk}$. 
\end{definition}

The following theorem shows that not all compact sets of stabilizable systems enable robust stabilization by \emph{static} state-feedback, revealing a fundamental limitation of this class of feedback laws for robust control. 

\begin{theorem}
\label{thm: static state-feedback}
There exists a compact set $\Sigma_\text{pk}\subseteq\Sigma_\text{stab}$ that does not enable robust stabilization by static state-feedback.  
\end{theorem}

As a consequence of Theorems \ref{th:dynamic} and \ref{thm: static state-feedback}, we have that for some uncertainty sets, robust stabilization requires a \emph{dynamic} feedback law. To prove Theorem \ref{thm: static state-feedback}, as an intermediate step, we first consider the scalar case $n=m=1$. 
\begin{lemma}
\label{lem: scalar static}
Suppose that $n=m=1$. Then, there exists a compact set $\Sigma_\text{s}\subset\Sigma_{\text{stab}}$ that does not enable robust stabilization by static state-feedback.
\end{lemma}
\begin{proof}
Take the compact set of stabilizable systems as
\begin{equation}
\label{eqn: SISO}
\Sigma_\text{s}\coloneqq\set{ (a,b) }{a=0, |b| \leq 1 } \cup \{ (2,1) \}.
\end{equation} 
Let $g:\mathbb{R}\rightarrow \mathbb{R}$. We show that there exists $(a,b)\in\Sigma_\text{s}$ such that \eqref{eq:def1} is not globally asymptotically stable. First, suppose that $g$ satisfies
\begin{equation} 
\label{eqn: f monotone decreasing 1}
|g(x)| < |x|\ \text{ for all }\ x\in\mathbb{R}\backslash\{0\}.
\end{equation}  
Now, take $a=2$ and $b=1$. The closed-loop system is  
\begin{equation}
\label{eqn:a=2,b=1 scalar}
x(t+1) = 2 x(t) + g(x(t)).
\end{equation}
Let $x(0)\neq 0$. It follows from \eqref{eqn: f monotone decreasing 1} and \eqref{eqn:a=2,b=1 scalar} that
\begin{equation}
|x(t+1)| = |2x(t)+g(x(t))| \geq \big| 2|x(t)| - |g(x(t))|\big| > |x(t)|
\end{equation}
for all $t\in\mathbb{Z}_+$. This implies that \eqref{eqn:a=2,b=1 scalar} is not globally asymptotically stable.  Hence, a static state-feedback $g$ satisfying \eqref{eqn: f monotone decreasing 1} cannot stabilize all systems in the set $\Sigma_\text{s}$. Second, consider a static state-feedback $g$ that does not satisfy \eqref{eqn: f monotone decreasing 1}, i.e., there exists some $x_0\neq 0$ such that 
\begin{equation}
\label{eqn: f not monotone scalar}
|g(x_0)| \geq |x_0|.
\end{equation}
Consider $a=0$ and $b = \frac{x_0}{g(x_0)}$. Note that $|b|\leq 1$ holds by \eqref{eqn: f not monotone scalar}, thus, we have $(a,b)\in\Sigma_\text{s}$. The closed-loop system is then given by     
\begin{equation}\label{eqn:a=0,b=sth scalar}
x(t+1) = \tfrac{x_0}{g(x_0)}g(x(t)).
\end{equation}
Take $x(0)=x_0$ and observe that $x(1)= \frac{x_0}{g(x_0)} g(x_0) = x_0$. Hence, the state trajectory satisfies $x(t)=x_0$ for all $t\in\mathbb{Z}_+$.  This implies that \eqref{eqn:a=0,b=sth scalar} is not globally asymptotically stable. 
\end{proof}

The proof of Theorem~\ref{thm: static state-feedback} now follows from Lemma~\ref{lem: scalar static}.

\textit{Proof of Theorem~\ref{thm: static state-feedback}:} Define \mbox{$\mathcal{A}:\mathbb{R}\rightarrow\mathbb{R}^{n\times n}$} and \mbox{$\mathcal{B}:\mathbb{R}\rightarrow\mathbb{R}^{n\times m}$} as $\mathcal{A}(a)=\begin{bmatrix}
ae_1 \!&\! 0
\end{bmatrix}$ and $\mathcal{B}(b) =\begin{bmatrix}
be_1 \!&\! 0
\end{bmatrix}$, where $e_1\in\mathbb{R}^n$ is the vector with its first entry equal to $1$ and the rest of entries equal to zero. Recall the definition of $\Sigma_\text{s}$ from \eqref{eqn: SISO}, and take $\Sigma_\text{pk}\coloneqq\set{ (\mathcal{A}(a),\mathcal{B}(b)) }{(a,b)\in\Sigma_\text{s}}$. Let $g:\mathbb{R}^n\rightarrow \mathbb{R}^n$. Define $x_1(t)\in\mathbb{R}$ and $g_1:\mathbb{R}\rightarrow \mathbb{R}$ as \mbox{$x_1(t)\coloneqq e_1^\top x(t)$} and $g_1(y)\coloneqq e_1^\top g(e_1 y)$. Observe from the closed-loop dynamics \eqref{eq:def1} that for every \mbox{$(A,B)\in\Sigma_\text{pk}$}, starting from \mbox{$x(0)=x_1(0)e_1$}, the state trajectory satisfies
\begin{equation}
\label{eqn:a,b}
x_1(t+1) = a x_1(t) + b g_1(x_1(t)).
\end{equation}
This implies that if the closed-loop system \eqref{eq:def1} is globally asymptotically stable for all $(A,B)\in\Sigma_\text{pk}$, then \eqref{eqn:a,b} is globally asymptotically stable for all $(a,b)\in\Sigma_\text{s}$. Based on Lemma~\ref{lem: scalar static}, this is not possible. Therefore, the set $\Sigma_\text{pk}$ does not enable robust stabilization by static state-feedback. \hfill \QED

\subsection{Fundamental limitation of linear dynamic state-feedback}
\label{sec:IV-lin}

In this section, we consider linear dynamic state-feedback laws of the form
\begin{equation} 
\label{eq:fd_lin_dyn}
\begin{split}
z(t+1)&=Kz(t)+Lx(t), \\
u(t)&=Mz(t) + Nx(t),
\end{split}
\end{equation}
where $p\!\in\!\mathbb{N}$, $K\!\in\!\mathbb{R}^{p\times p}$, $L\!\in\!\mathbb{R}^{p\times n}$, $M\!\in\!\mathbb{R}^{m\times p}$, and $N\!\in\!\mathbb{R}^{m\times n}$. The closed-loop system interconnecting \eqref{eq:1} and \eqref{eq:fd_lin_dyn} is
\begin{equation}
\label{eq:fd_lin_dyn_cl}
\begin{split}
x(t+1)&=(A+BN)x(t)+BMz(t),\\
z(t+1)&=Lx(t)+Kz(t).
\end{split}
\end{equation}
By $x(t,x_0,z_0)$ and $z(t,x_0,z_0)$, respectively, we denote the states of \eqref{eq:def2} resulting from the initial conditions $x(0)=x_0$ and $z(0)=z_0$.
We investigate whether for any compact set $\Sigma_\text{pk}\subseteq \Sigma_\text{stab}$ there exists a robust dynamic state-feedback that is linear. To formalize this, we consider a stronger version of Definition~\ref{def:robust dynamic stab} that requires the feedback to be linear. 

\begin{definition}
\label{def: linear dynamic stab}
We say that $\Sigma_\text{pk}$ \emph{enables robust stabilization by linear dynamic state-feedback} if there exists $p\in\mathbb{N}$, $K\in\mathbb{R}^{p\times p}$, $L\in\mathbb{R}^{p\times n}$, $M\in\mathbb{R}^{m\times p}$, and $N\in\mathbb{R}^{m\times n}$ such that the following statements hold:
\begin{enumerate}[label=(\alph*),ref=\ref{th:12}(\alph*)]
    \item System \eqref{eq:fd_lin_dyn_cl} is globally asymptotically stable with respect to $x$ for all $(A,B)\in\Sigma_\text{pk}$. 
    \item For every $x_0\in\mathbb{R}^n$, $z_0\in\mathbb{R}^p$, and $(A,B)\in\Sigma_\text{pk}$, the controller's state $\|z(t,x_0,z_0)\|$ is bounded. 
\end{enumerate}
\end{definition}

The following theorem shows that not all compact sets of stabilizable systems enable robust stabilization by linear dynamic state-feedback. Similar observations in the frequency-domain setting have already been made, e.g., in \cite{vidyasagar1986robust}. Nevertheless, for the sake of completeness, we provide a full proof of the result.

\begin{theorem}
\label{thm: dynamic linear}
There exists a compact set $\Sigma_\text{pk}\subseteq\Sigma_\text{stab}$ that does not enable robust stabilization by linear dynamic state-feedback. 
\end{theorem}
\begin{proof}
Let $e_{1,r}\in\mathbb{R}^r$ with $r\in\mathbb{N}$ denote a vector with the first entry equal to $1$ and the rest of the entries equal to zero. Let $\hat{A}=2e_{1,n} e_{1,n}^\top$ and $\hat{B}=e_{1,n} e_{1,m}^\top$. Take the compact set of stabilizable systems as $\Sigma_\text{pk} \coloneqq \{(\hat{A},\hat{B}),(\hat{A},-\hat{B}) \}$. We show that no linear dynamic state-feedback simultaneously stabilizes both systems in the set $\Sigma_\text{pk}$. Aiming for a contradiction, suppose that $K$, $L$, $M$, and $N$ are given such that the closed-loop system \eqref{eq:fd_lin_dyn_cl} satisfies conditions (a) and (b) in Definition~\ref{def: linear dynamic stab} for all $(A,B)\in\Sigma_\text{pk}$. Since the closed-loop system is linear, conditions (a) and (b) in Definition~\ref{def: linear dynamic stab} imply that the trajectories of the state $x(t,x_0,z_0)$ and the controller's state $z(t,x_0,z_0)$ are bounded for all \mbox{$x_0\in\mathbb{R}^n$}, $z_0\in\mathbb{R}^n$, and $t\in\mathbb{Z}_+$. 
We define $F(A,B)\coloneqq\begin{bmatrix} A+BN & BM \\ L & K \end{bmatrix}$. It is then implied that the eigenvalues of both matrices $F(\hat{A},\hat{B})$ and $F(\hat{A},-\hat{B})$ lie in the closed unit disk. Thus, $\det(2I-F(\hat{A},\pm \hat{B}))>0$. By multilinearity of the determinant, we have
\begin{align}
    &\det(2I-F(\hat{A}, \hat{B})) + \det(2I-F(\hat{A},-\hat{B})) \\
    &= 2 \det \begin{bmatrix} 2I\!-\!\hat{A} & 0 \\ -L & 2I\!-\!K \end{bmatrix} = 2\det(2I-\hat{A})\det(2I-K) = 0,
\end{align}
which contradicts that $\det(2I-F(\hat{A},\pm \hat{B}))>0$.
\end{proof}

Theorems~\ref{thm: static state-feedback} and~\ref{thm: dynamic linear} show that to simultaneously stabilize all systems within an arbitrary compact set of stabilizable systems, one must consider feedback laws that are both dynamic and nonlinear. 

%% file: design.tex
\section{Polytopic Set Covering}
\label{sec:VI}
It was shown in Section~\ref{sec:III} that if $\Sigma_\text{pk}$ enables robust \mbox{$\alpha$--exponential} stabilization, a feedback law as in \eqref{eq:fun_fg} simultaneously stabilizes all systems in $\Sigma_\text{pk}$ with an exponential decay rate strictly less than $\alpha$. A key step towards designing such a feedback law is to cover the set $\Sigma_\text{pk}$ by a finite number of sets that are, individually, \mbox{$\alpha$--quadratically} stabilizable. That is, the problem is to find $q\in\mathbb{N}$, and $K_i$ and $P_i>0$ for $i\in\{1,\ldots,q\}$, such that for every $(A,B)\in\Sigma_\text{pk}$ we have $\alpha^2 P_i-(A+BK_i)^\top P_i(A+BK_i)\geq I$ for some $i\in\{1,\ldots,q\}$. In this section, we focus on polytopic sets and provide an algorithm by which one can compute $q\in\mathbb{N}$, and $K_i$ and $P_i>0$ for $i\in\{1,\ldots,q\}$, satisfying the above properties\footnote{The results of this section can be extended to a union of polytopic sets.}. 

Consider polytopic sets with an affine parametrization of the form
\begin{equation}
\label{eq:sigma_aff}
\Sigma_\text{pk}=\set{(\mathcal{A}(\theta),\mathcal{B}(\theta))}{\theta_\text{min}\preceq\theta\preceq \theta_\text{max}},
\end{equation}
where $\theta_\text{min},\theta_\text{max}\in\mathbb{R}^\ell$ with $\theta_\text{min}\preceq \theta_\text{max}$. Moreover, the pair \mbox{$(\mathcal{A}(\theta),\mathcal{B}(\theta)):\mathbb{R}^\ell\rightarrow \Sigma$} is defined as
\begin{equation}
\label{eq:aff}
\mathcal{A}(\theta)\coloneqq A_0+\sum_{i=1}^\ell\theta_iA_i\ \text{ and }\ \mathcal{B}(\theta)\coloneqq B_0+\sum_{i=1}^\ell\theta_iB_i,
\end{equation}
with $\theta_i$ denoting the $i$th entry of $\theta$. The integer $\ell$, the pairs $(A_i,B_i)$ for $i\in\{0,\ldots,\ell\}$, and the vectors $\theta_\text{min},\theta_\text{max}$ are given parameters. This is a common class of uncertainty sets studied in robust control (e.g., see \cite{feron2002analysis,scherer1999lecture,amato2002note}). We note that any compact set can be over-approximated by a polytopic set of this form. In this formulation, the parameter $\theta$ represents the system uncertainties. This parameter belongs to the hyper-rectangle defined by $\theta_\text{min}\preceq\theta\preceq\theta_\text{max}$. We denote the set of vertices of this hyper-rectangle by \mbox{$\mathcal{V}(\theta_\text{min},\theta_\text{max})\subset\mathbb{R}^\ell$}, which has $2^\ell$ elements. 

For this class of $\Sigma_\text{pk}$, the following lemma provides a full characterization of their $\alpha$--quadratic stabilizability. This result extends \cite[Thm. 1]{amato2002note} by taking the parameter $\alpha$ into account. 



\begin{proposition}
\label{prop:lmi_rectangle_our}
Let $\alpha\in[0,1)$. Suppose that $\Sigma_\text{pk}$ is given by \eqref{eq:sigma_aff} and it satisfies $\Sigma_\text{pk}\subset\Sigma_\text{stab}(\alpha)$. Then, $\Sigma_\text{pk}$ is \mbox{$\alpha$--quadratically} stabilizable \emph{if and only if} there exists $Q>0$ and $L$ such that
\begin{equation}
\label{eqn: LMI vertex}
\begin{bmatrix}
\alpha^2 Q & (\mathcal{A}(\theta)Q+\mathcal{B}(\theta)L)^\top & Q \\
\mathcal{A}(\theta)Q+\mathcal{B}(\theta)L & Q & 0 \\
Q & 0 & I
\end{bmatrix}\geq0
\end{equation}
for all $\theta\in\mathcal{V}(\theta_\text{min},\theta_\text{max})$. Moreover, if \eqref{eqn: LMI vertex} holds, then \mbox{$K=LQ^{-1}$} and $P=Q^{-1}$ are such that for every \mbox{$(A,B)\in\Sigma_\text{pk}$} we have \mbox{$\alpha^2 P-(A+BK)^\top P(A+BK)\geq I$}.
\end{proposition}
\begin{proof}
Define $\Theta\coloneqq\set{\theta\in\mathbb{R}^\ell}{\theta_\text{min}\preceq\theta\preceq\theta_\text{max}}$. It follows from Lemma~\ref{lem:QS_compact} that $\Sigma_\text{pk}$ is $\alpha$--quadratically stabilizable if and only if there exists $P\geq I$ and $K$ such that 
\begin{equation}
\label{eqn: ineq vertex_2}
\alpha^2 P-(\mathcal{A}(\theta)+\mathcal{B}(\theta)K)^\top P(\mathcal{A}(\theta)+\mathcal{B}(\theta)K)\geq I
\end{equation}
for all $\theta\in\Theta$. Let $Q=P^{-1}>0$ and \mbox{$L=KQ$}. Multiply \eqref{eqn: ineq vertex_2} from left and right by $Q$ to obtain
\begin{equation}
\label{eqn: ineq vertex_3}
\alpha^2 Q-(\mathcal{A}(\theta)Q+\mathcal{B}(\theta)L)^\top Q^{-1}(\mathcal{A}(\theta)Q+\mathcal{B}(\theta)L)\geq Q^2.
\end{equation}
One can verify using a Schur complement argument that \eqref{eqn: ineq vertex_3} is equivalent to \eqref{eqn: LMI vertex}. Therefore, $\Sigma_\text{pk}$ is \mbox{$\alpha$--quadratically} stabilizable if and only if there exists $Q>0$ and $L$ such that \eqref{eqn: LMI vertex} holds for all \mbox{$\theta\in\Theta$}. Now, the ``only if'' part of Proposition~\ref{prop:lmi_rectangle_our} is obvious since if \eqref{eqn: LMI vertex} holds for all $\theta\in\Theta$ then it also hold for all  $\theta\in\mathcal{V}(\theta_\text{min},\theta_\text{max})\subset\Theta$. For the ``if'' part, observe that $\Theta$ is the convex hull of $\mathcal{V}(\theta_\text{min},\theta_\text{max})$. One can verify that if \eqref{eqn: LMI vertex} holds for $\theta_1,\theta_2\in\mathcal{V}(\theta_\text{min},\theta_\text{max})$, then it also holds for $\theta=\beta\theta_1+(1-\beta)\theta_2$ with any $\beta\in[0,1]$ due to \eqref{eq:aff}. Therefore, \eqref{eqn: LMI vertex} holds for all $\theta\in\Theta$. 
\end{proof}


Given $\Sigma_\text{pk}$ of the form \eqref{eq:sigma_aff}, 
we leverage LMI~\eqref{eqn: LMI vertex} to cover $\Sigma_\text{pk}$ by finitely many polytopic sets that are, individually, $\alpha$--quadratically stabilizable. This can be accomplished using Algorithm~\ref{alg:1} presented next.

\begin{algorithm}[ht]
\caption{Polytopic set covering}
\label{alg:1}
\begin{algorithmic}[1]
\Require{$\theta_\text{min}$, $\theta_\text{max}$, and $\alpha$}
\State{$j\gets 1$, $q\gets 0$, $\theta_\text{min}^{(1)}\gets\theta_\text{min}$, $\theta_\text{max}^{(1)}\gets\theta_\text{max}$, and $\mathcal{I}\gets\varnothing$} \label{step1}
\While{$q<j$}
   \State{$q=j$}
   \For{$i\in\{1,\ldots,q\}\backslash\mathcal{I}$}
   
   \State{Let $\mathcal{V}_i=\mathcal{V}(\theta_\text{min}^{(i)},\theta_\text{max}^{(i)})$.}
   \If{$Q>0$ and $L$ satisfy \eqref{eqn: LMI vertex} for all \mbox{$\theta\in\mathcal{V}_i$}}\label{alg:1_step1}
   \State{$P_i\gets Q^{-1}$, $K_i\gets LQ^{-1}$, and $\mathcal{I}\gets \mathcal{I}\cup\{i\}$}
   \Else
   \State{$\hat{\theta}\gets\frac{1}{2}(\theta_\text{min}^{(i)}+\theta_\text{max}^{(i)})$ and $\theta_\text{max}^{(j+1)}\gets \theta_\text{max}^{(i)}$}\label{step15}
   \State{Let $k\in\arg\max_\kappa (\theta_{\text{max},\kappa}^{(i)}-\hat{\theta}_\kappa)$}
   \State{$\theta_{\text{min},\kappa}^{(j+1)}\gets \theta_{\text{min},\kappa}^{(i)}$ if $\kappa\neq k$ and $\hat{\theta}_\kappa$ if $\kappa=k$}
   \State{$\theta_{\text{max},k}^{(i)}\gets \hat{\theta}_k$ and $j\gets j+1$}\label{step20}
   \EndIf
   \EndFor
\EndWhile
\Ensure{$q$, and $K_i,P_i$ for all $i\in\{1,\ldots,q\}$}
\end{algorithmic}
\end{algorithm}

Algorithm~\ref{alg:1} works as follows.  
We start by considering the polytopic set of the form~\eqref{eq:sigma_aff} as a whole.
We then proceed to check if this set is $\alpha$--quadratically stabilizable by verifying if there exist $Q>0$ and $L$ which satisfy \eqref{eqn: LMI vertex} for all its vertices. If so, the algorithm terminates. If not, we divide the set in half along its longest edge. Next, those two subsets are considered separately. If either one of the subsets is $\alpha$--quadratically stabilizable, we store its label in the set $\mathcal{I}$ and no longer consider it in the remainder of the algorithm's iterations.
For the remaining subset(s), that is, the subsets whose label is not in $\mathcal{I}$, we continue the same procedure until all the subsets are $\alpha$-quadratically stabilizable. It is easy to see that this algorithm will terminate within a finite number of steps as long as $\rho_\alpha(\Sigma_\text{pk})>0$, which is the case if $\Sigma_\text{pk}\subset\Sigma_\text{stab}(\alpha)$ and $\Sigma_\text{pk}$ is compact (see the proof of Lemma~\ref{lem:finite covering quadratic stab sets}).

%% file: examples.tex
\section{Numerical Example} \label{sec: examples}
Consider the two-wheeled inverted pendulum mobile platform\footnote{The simulation code is available at https://github.com/a-shakouri/robust-control-using-nonlinear-feedback.} in Fig~\ref{fig:twowheel}. The state variables $x_1(t)$, $x_3(t)$, and $x_5(t)$ denote the platform's position, pitch angle, and yaw angle, respectively. The rest of the state variables $x_2(t)$, $x_4(t)$, and $x_6(t)$ denote the rate of change of $x_1(t)$, $x_3(t)$, and $x_5(t)$, respectively. The control inputs $u_1(t)$ and $u_2(t)$ are the torques applied to the platform by each of the wheels. The linearized state-space equations of this system around the unstable upright position are provided in \cite[Sec. III-B]{kim2017nonlinear}. We discretize this dynamics using the forward Euler method with a time step $0.01$. The nominal system is 
\begin{equation}
A_0=\begin{bmatrix}
1 & 0.01 & 0 & 0 & 0 & 0 \\
0 & 1 & a_1 & 0 & 0 & 0 \\
0 & 0 & 1 & 0.01 & 0 & 0 \\
0 & 0 & a_2 & 1 & 0 & 0 \\
0 & 0 & 0 & 0 & 1 & 0.01 \\
0 & 0 & 0 & 0 & 0 & 1
\end{bmatrix},\ B_0=\begin{bmatrix}
0 & 0 \\
b_1 & b_1 \\
0  &  0 \\
b_2 & b_2 \\
0 & 0 \\
b_3 & -b_3 
\end{bmatrix},
\end{equation}
with parameters $a_1=-0.0332$, $a_2=0.2732$, $b_1=0.0019$, \mbox{$b_2=-0.0073$}, and $b_3=-0.0085$. We assume that the true values of these parameters can differ up to $\pm 5\%$ of their nominal values. To capture this uncertainty, we consider $\Sigma_\text{pk}$ in the form of \eqref{eq:sigma_aff} with $\ell=5$, $A_3=A_4=A_5=0$, $A_1=\begin{bmatrix}
0_{6\times 2} & e_2 & 0_{6\times 3}\end{bmatrix}$, \mbox{$A_2=\begin{bmatrix}
0_{6\times 2} & e_4 & 0_{6\times 3}\end{bmatrix}$}, $B_1=B_2=0$, $B_3=\begin{bmatrix}
e_2 & e_2
\end{bmatrix}$, \mbox{$B_4=\begin{bmatrix}
e_4 & e_4
\end{bmatrix}$}, and $B_5=\begin{bmatrix}
e_6 & -e_6
\end{bmatrix}$. Here $e_i\in\mathbb{R}^6$ denotes the vector whose $i$th entry is 1 and all other entries are zero. Moreover, we set $\theta_\text{max}=0.05\begin{bmatrix}
-a_1 \!&\! a_2 \!&\! b_1 \!&\! -b_2 \!&\! -b_3
\end{bmatrix}^\top$ and $\theta_\text{min}=-\theta_\text{max}$. 

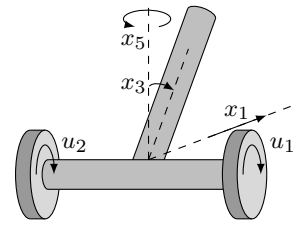
\begin{figure}
    \centering
    \begin{tikzpicture}[scale=0.39]
    \newcommand{\rod}[3]{%
    \begin{scope}[shift={(#2,#3)}, rotate=#1]
    \fill[gray!50,rotate=-90] (-0.5,0.6) rectangle (0.5,6.53);
    \fill[gray!50,rotate=-90] (0,0.6) ellipse (0.5 and 0.2);
    \draw[,rotate=-90] (-0.5,0.6) -- (-0.5,6.53);
    \draw[,rotate=-90] (0.5,0.6) -- (0.5,6.53);
    \draw[,rotate=-90] (-0.5,0.6) arc (180:360:0.5 and 0.2); 
    \end{scope}
    }
    \newcommand{\wheel}[3]{%
    \begin{scope}[shift={(#2,#3)}, rotate=#1]
    \fill[gray!80,rotate=-90] (-1.5,0) rectangle (1.5,0.5);
    \fill[gray!80,rotate=-90] (0,0) ellipse (1.5 and 0.5);
    \filldraw[fill=gray!30, draw=black,rotate=-90] (0,0.5) ellipse (1.5 and 0.5);
    \draw[,rotate=-90] (-1.5,0) -- (-1.5,0.5);
    \draw[,rotate=-90] (1.5,0) -- (1.5,0.5);
    \draw[,rotate=-90] (-1.5,0) arc (180:360:1.5 and 0.5); 
    \end{scope}
    }
    \rod{250}{6}{6}
    \wheel{0}{0}{0}
    \rod{0}{0}{0}
    \wheel{0}{7}{0}
    \draw[{latex}-,rotate=90] (0,-0.8) arc (-90:90:1 and 0.3); 
    \node at (1.5,1) {$u_2$};
    \draw[{latex}-,rotate=90] (0,-7.8) arc (-90:90:1 and 0.3); 
    \node at (8.55,1) {$u_1$};
    \draw[dashed] (4,0.5)--(4,6); 
    \draw[-{latex}] (4.35,5) arc (-60:240:0.8 and 0.3); 
    \node at (3.5,4.5) {$x_5$};
    \draw[-{latex},rotate=150] (-2,-4.6) arc (0:-90:0.8 and 0.3); 
    \node at (3.5,3) {$x_3$};
    \draw[dashed] (4,0.5) -- ($(4,0.5)!1.5!(4.9,3)$);
    \draw[dashed] (4,0.5) -- ($(4,0.5)!1.25!(8,2)$);
    \draw[{latex}-] (8,2) -- ($(4,0.5)!0.5!(8,2)$);
    \node at (7,2.1) {$x_1$};
\end{tikzpicture}
    \caption{Schematic of a two-wheeled inverted pendulum.}
    \label{fig:twowheel}
\end{figure}

We let $\alpha=0.9$ and use Algorithm~\ref{alg:1} to cover $\Sigma_\text{pk}$ by \mbox{$q=108$} subsets. We apply the feedback law \eqref{eq:fun_fg} to simultaneously stabilize all systems within $\Sigma_\text{pk}$. Figure~\ref{fig:numerical ex} shows the closed-loop trajectory of the particular system $(\mathcal{A}(\theta),\mathcal{B}(\theta))$ with \mbox{$\theta = \theta_\text{min} + \tfrac{1}{10}(\theta_\text{max}-\theta_\text{min})$}.
The initial conditions are
$$x(0) =0.01\begin{bmatrix}
    -2 \!&\! 0.5 \!&\! 1 \!&\! 3 \!&\! 0.2 \!&\! 7
\end{bmatrix}^\top \text{ and } z(0)=\begin{bmatrix}
0 \!&\! 0
\end{bmatrix}^\top.$$
From the bottom plot of Fig.~\ref{fig:numerical ex}, one can see that $23$ subsets were explored and falsified by the feedback law before it settled on \mbox{$z_2(t)=24$} for all $t\geq3.1$. Now, based on the history of $z_2(t)$ from $t=0$ to $t=3.1$, one can tell that the true system $(\mathcal{A}(\theta),\mathcal{B}(\theta))$ does not belong to the subsets of $\Sigma_\text{pk}$ that are labeled $1$ to $23$ in the set covering algorithm. Depending on the true system and the initial conditions, more or fewer subsets can be falsified. 

\begin{figure}[h]
    \centering
    \includegraphics[width=1\linewidth]{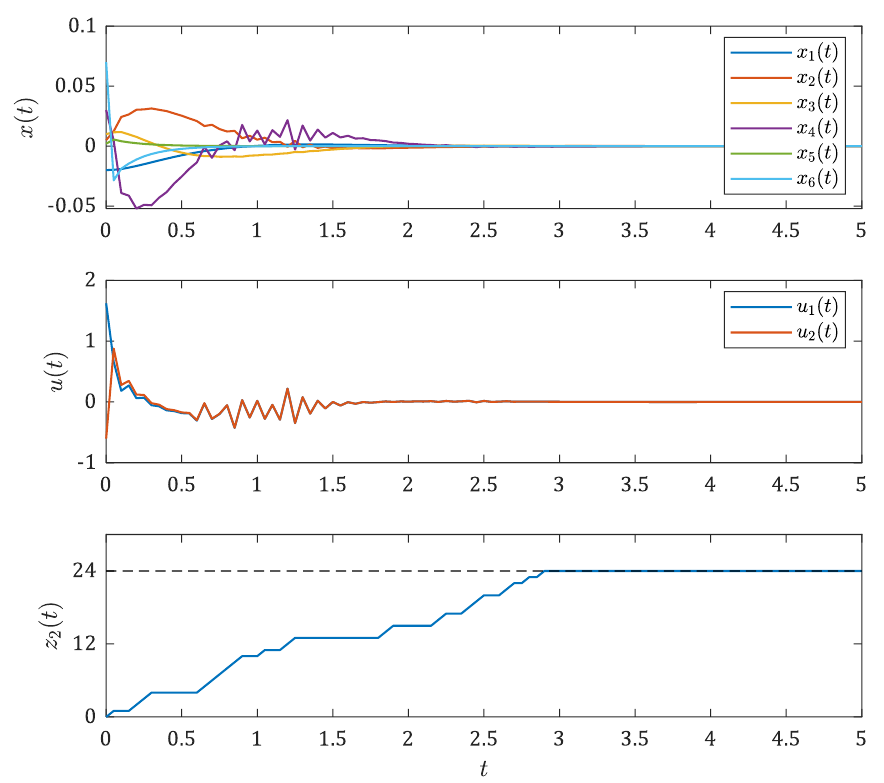}
    \caption{Closed-loop trajectory for the two-wheeled robot.}
    \label{fig:numerical ex}
\end{figure}

%% file: conclusion.tex
\section{Conclusion} \label{sec: conclusion}
In this paper, we have studied the robust stabilization of compact sets of linear time-invariant systems. We have shown that, in general, robust stabilization \emph{requires} a state-feedback that is both nonlinear and dynamic. Moreover, we have shown that, for \emph{any} compact set of stabilizable systems, one can find a (time-invariant) nonlinear dynamic state-feedback that simultaneously stabilizes all systems within the set. 
In addition to asymptotic stabilization, we have also investigated exponential stabilization with a given rate of decay. 


In this work, we studied robust stabilization by only considering uncertainties in the system parameters. In practice, the system might face several types of uncertainty, e.g., process noise. In particular, the study of input-to-state stability of such robust dynamic feedback laws with respect to an exogenous noise signal is an interesting topic for future work. 

%% file: appendix.tex
\appendix

\subsection{Proof of Lemma~\ref{lem:alpha_Hautus}} \label{app:proof of alpha_Hautus}
(a)$\Rightarrow$(b): We use contraposition. Suppose that (b) does not hold. Then, there exist nonzero $v\in\mathbb{C}^n$ and $\lambda\in\mathbb{C}$ with $|\lambda|\geq\alpha$ such that $v^*B=0$ and $v^*A=\lambda v^*$. Let $K\in\mathbb{R}^{m\times n}$. We have that $v^* (A+BK)=\lambda v^*$. Therefore, \mbox{$\lambda\in\spec (A+BK)$} satisfies $|\lambda|\geq\alpha$. Since this argument holds for all \mbox{$K\in\mathbb{R}^{m\times n}$}, the pair $(A,B)$ is not $\alpha$--stabilizable. Therefore, (a) does not hold. (b)$\Rightarrow$(c): Suppose that (b) holds, i.e., every eigenvalue of $A$ that is not $(A,B)$--controllable has modulus strictly less than $\alpha$. It follows from \cite[Thm. 3.32]{trentelmancontrol} that there exists $K\in\mathbb{R}^{m\times n}$ such that every member of $\spec (A+BK)$ has modulus strictly less than $\alpha$. This implies that $\frac{1}{\alpha}(A+BK)$ is Schur. Thus, there exists $P>0$ such that the Lyapunov inequality \mbox{$P-\frac{1}{\alpha^2}(A+BK)^\top P(A+BK) >0$} is satisfied. Multiplying this inequality by $\alpha^2$ gives \eqref{eq:lyap}. (c)$\Rightarrow$(a): Suppose that (c) holds. Let $P>0$ and $K$ satisfy \eqref{eq:lyap}. Let $v\in\mathbb{C}^n\backslash\{0\}$ and $\lambda\in\mathbb{C}$ be such that $v^*(A+BK)=\lambda v^*$. Multiply \eqref{eq:lyap} from left and right, respectively, by $v^*$ and $v$ to obtain $(\alpha^2-|\lambda|^2)v^* Pv>0$. Since $v^* Pv>0$, this implies that $|\lambda|<\alpha$. Therefore, $(A,B)$ is $\alpha$--stabilizable. \hfill \QED

\subsection{Proof of Lemma~\ref{lem:QS_compact}} \label{app:proof of QS_compact}
The ``if'' part is obvious since \eqref{eq:Lyap_P^2} and \mbox{$P\geq I$} imply~\eqref{eq:lyap}. For the ``only if'' part, suppose that $\mathcal{S}$ is \mbox{$\alpha$--quadratically} stabilizable. Let $\bar{P}>0$ and $K$ be such that \mbox{$\alpha^2 \bar{P}-(A+BK)^\top \bar{P}(A+BK) >0$} holds for all \mbox{$(A,B)\in\mathcal{S}$}. We claim that there exists a sufficiently small $\gamma>0$ such that for every $(A,B)\in\mathcal{S}$ we have \mbox{$\alpha^2\bar{P}-(A+BK)^\top \bar{P}(A+BK)\geq \gamma I$}.
To prove this claim, suppose on the contrary that such a $\gamma$ does not exist. Thus, there exists a sequence $(A_i,B_i)_{i=0}^\infty$ in $\mathcal{S}$ such that
\begin{equation}
\label{eq:lim_alpha_Lyap}
\lim_{i\rightarrow\infty}\lambda_\text{min}(\alpha^2 \bar{P}-(A_i+B_iK)^\top \bar{P}(A_i+B_iK))=0.
\end{equation}
Since $\mathcal{S}$ is compact, it follows from the Bolzano–Weierstrass theorem that $(A_i,B_i)_{i=0}^\infty$ has a convergent subsequence with limit, say, $(A_*,B_*)\in\mathcal{S}$. Since the smallest eigenvalue of a matrix is a continuous function of its entries (see \cite[pp. 125--126]{kato2012short}), \eqref{eq:lim_alpha_Lyap} implies that $\lambda_\text{min}(\alpha^2 \bar{P}-(A_*+B_*K)^\top \bar{P}(A_*+B_*K))=0$.
This contradicts the assumption that every $(A,B)\in\mathcal{S}$ satisfies \mbox{$\alpha^2 \bar{P}-(A+BK)^\top \bar{P}(A+BK) >0$}. Therefore, the claim holds. Now, take $P=\frac{1}{\gamma} \bar{P}$ and verify that \eqref{eq:Lyap_P^2} holds. Here, observe that \eqref{eq:Lyap_P^2} implies $\alpha^2 P\geq I$. Since $0<\alpha\leq 1$, this implies that $P\geq \frac{1}{\alpha^2} I\geq I$, concluding the proof. \hfill \QED

\subsection{Proof of Lemma~\ref{lem:properties}} \label{app: proof of lem sdp}

For this proof, we first need the following auxiliary result. 

\begin{lemma}
\label{lem:nil}
Let $\alpha\in(0,1]$, $M\in\mathbb{R}^{n\times n}$, and $\sigma$ denote the smallest singular value of $\alpha I-M$. If $M$ is nonzero and nilpotent, then $\sigma<\alpha$. 
\end{lemma}
\begin{proof}
Since all eigenvalues of $\alpha I-M$ are equal to $\alpha$, by Browne's inequality \cite[p. 16-2]{hogben2006handbook},  we have $\sigma\leq\alpha$. Now, to show that $\sigma<\alpha$, suppose on the contrary that $\sigma=\alpha$. Since $\det(\alpha I-M)=\alpha^n$, the product of all singular values is equal to $\alpha^n$. This, together with $\sigma=\alpha$, implies that all singular values of $\alpha I-M$ are equal to $\alpha$. Thus, $\tr ((\alpha I-M)(\alpha I-M)^\top)=n\alpha^2$. Since $M$ is nilpotent, this implies $M=0$. Thus, we reach a contradiction and $\sigma< \alpha$. 
\end{proof}

\textit{Proof of Lemma~\ref{lem:properties}:} (a) It follows by definition that \mbox{$\rho_\alpha(\hat{A},\hat{B})\geq0$}. Now, to show that $\rho_\alpha(\hat{A},\hat{B})\leq\alpha$, we suppose on the contrary that $\rho_\alpha(\hat{A},\hat{B})>\alpha$. This implies that $\mathcal{B}_r(\hat{A},\hat{B})$ is $\alpha$--quadratically stabilizable for some $r>\alpha$. Note that $(\hat{A}+\beta I,\hat{B})\in\mathcal{B}_r(\hat{A},\hat{B})$ for all $\beta\in[-r,r]$. Now, since $\mathcal{B}_r(\hat{A},\hat{B})$ is $\alpha$--quadratically stabilizable, due to Lemma~\ref{lem:QS_compact}, there exists $P\geq I$ and $K$ such that
\begin{equation}
\label{eq:Lyap_beta-r}
\alpha^2 P-(\hat{A}+\beta I+\hat{B}K)^\top P(\hat{A}+\beta I+\hat{B}K)\geq I
\end{equation}
for all $\beta\in[-r,r]$. Let $\lambda\in \mathbb{C}$ and $v\in\mathbb{C}^n$ with $\|v\|=1$ be such that $(\hat{A}+\hat{B}K)v=\lambda v$. Let $\beta=\frac{|\Re(\lambda)|}{\Re(\lambda)}r$ if $\Re(\lambda)\neq0$ and $\beta=r$ if $\Re(\lambda)=0$. Multiply \eqref{eq:Lyap_beta-r} from left and right, respectively, by $v^*$ and $v$ to obtain $(\alpha^2-|\lambda+\beta|^2) v^* Pv\geq 1$. Since $P\geq I$, this implies that $\alpha>|\lambda+\beta|$. Due to $r>\alpha$, this inequality does not hold. Therefore, we reach a contradiction, proving that $\rho_\alpha(\hat{A},\hat{B})\leq\alpha$.


(b) To prove the ``if'' part, suppose that $\hat{A}=0$. From part (a), we have \mbox{$\rho_\alpha(0,\hat{B})\leq\alpha$}. To prove that $\rho_\alpha(0,\hat{B})=\alpha$, it suffices to show that \mbox{$\rho_\alpha(0,\hat{B})\geq\alpha$}, i.e., $\mathcal{B}_r(0,\hat{B})$ is $\alpha$--quadratically stabilizable for all $r\in[0,\alpha)$. For this, let $r\in[0,\alpha)$, $K=0$, and $P=\frac{1}{\alpha^2-r^2}I$. We show that \eqref{eq:Lyap_P^2} holds for all members of $\mathcal{B}_r(0,\hat{B})$. Let $(A,B)\in\mathcal{B}_r(0,\hat{B})$ and note that $A^\top A\leq r^2 I$. Hence, \eqref{eq:Lyap_P^2} holds as $\tfrac{1}{\alpha^2-r^2}(\alpha^2I-A^\top A)\geq I$. To prove the ``only if'' part, suppose that $\rho_\alpha(\hat{A},\hat{B})=\alpha$. There exists $P\geq I$ and $K$ such that
\begin{equation}
\label{eq:Lyap_Delta_A Delta_B}
\alpha^2 P-(M_K+\Delta_A+\Delta_BK)^\top P(M_K+\Delta_A+\Delta_BK)\geq I
\end{equation}
for all $(\Delta_A,\Delta_B)\in\Sigma$ with $\norm{\begin{bmatrix}
\Delta_A & \Delta_B
\end{bmatrix}}<\alpha$.
Here, we denote $M_K\coloneqq\hat{A}+\hat{B}K$. First, we will show that $M_K$ is nilpotent.
Hence, in particular, \mbox{$\alpha^2 P-(M_K+\beta I)^\top P(M_K+\beta I)\geq I$} for all \mbox{$\beta\in(-\alpha,\alpha)$}. Let $\lambda\in \mathbb{C}$ and $v\in\mathbb{C}^n$ be such that $\|v\|=1$ and $M_Kv=\lambda v$. Taking the same steps as in the proof of part (a), one can verify that $\alpha> |\lambda+\beta|$ for all $\beta\in(-\alpha,\alpha)$. This implies that $\lambda=0$, thus, $M_K$ is nilpotent. Next, we will show that $M_K=0$. Aiming for a contradiction, suppose that $M_K\neq 0$. Let $\sigma\geq 0$ denote the smallest singular value of $\alpha I-M_K$. It follows from Lemma~\ref{lem:nil} that $\sigma<\alpha$. Now, let $x,y\in\mathbb{R}^n$ be such that $(\alpha I-M_K)y=\sigma x$ and $\|x\|=\|y\|=1$. Take $\Delta_A=\sigma xy^\top$ and observe that $\|\Delta_A\|= \sigma<\alpha$ and $(M_K+\Delta_A)y=\alpha y$. Take $\Delta_B=0$ and multiply \eqref{eq:Lyap_Delta_A Delta_B} from the left and right, respectively, by $y^\top$ and $y$ to reach a contradiction. Therefore, $M_K=0$. Finally, to prove that $\hat{A}=0$, it suffices now to show that $K=0$. For this, suppose on the contrary that $K\neq 0$. Let $v,w\in\mathbb{R}^n$ be such that $Kv=\|K\|w$ and $\|v\|=\|w\|=1$. Take $\Delta_A=\frac{\alpha}{1+\|K\|^2}vv^\top$ and $\Delta_B=\frac{\alpha\|K\|}{1+\|K\|^2}vw^\top$. One can verify that $\norm{\begin{bmatrix}
\Delta_A & \Delta_B
\end{bmatrix}}<\alpha$. We observe that $(\Delta_A+\Delta_B K)v=\alpha v$. Now, multiply \eqref{eq:Lyap_Delta_A Delta_B} from the left and right, respectively, by $v^\top$ and $v$ and substitute $M_K=0$ to reach a contradiction. Therefore, $K=0$ and thus $\hat{A}=0$.

(c) The ``if'' part follows from the definition of $\rho_\alpha(\hat{A},\hat{B})$. To prove the ``only if'' part, we use contraposition. Suppose that $(\hat{A},\hat{B})$ is $\alpha$--stabilizable. By Lemma~\ref{lem:alpha_Hautus}, there exist $P>0$, $K$, and $\gamma>0$ such that \mbox{$\alpha^2P-(\hat{A}+\hat{B}K)^\top P(\hat{A}+\hat{B}K)> \gamma I>0$}. Thus, there exists a sufficiently small $r>0$ such that \mbox{$\alpha^2P-(A+BK)^\top P(A+BK)> 0$} for all \mbox{$(A,B)\in\mathcal{B}_r(\hat{A},\hat{B})$}. Therefore, we have $\rho_\alpha(\hat{A},\hat{B})\geq r>0$. \hfill \QED

\subsection{Proof of Theorem~\ref{th:dynamic_cont}}
\label{app: proof of th dynamic_cont}

First, we will show that property \ref{(P4)} holds for $0<\alpha\leq 1$. To facilitate the proof, we define 
\begin{equation}
\kappa\coloneqq \max\set{\norm{A\!+\!BK_{i}}}{(A,B)\in\Sigma_\text{pk},i\in\{1,\ldots,q\}}, 
\end{equation}
$\lambda_\text{max}\coloneqq \max_{i}\lambda_\text{max}(P_i)$, and $\lambda_\text{min}\coloneqq \min_i\lambda_\text{min}(P_i)$. For simplicity, in this proof, we refer to $x(t,x_0,z_0)$ and $z(t,x_0,z_0)$, respectively, by $x(t)$ and $z(t)=\begin{bmatrix}
z_1(t) & z_2(t)
\end{bmatrix}^\top$. We define the set $\mathcal{T}\subset\mathbb{Z}_+$ as $\mathcal{T}\coloneqq\set{t\geq 1}{x(t)^\top P_{z_2(t)} x(t) > z_1(t)}$. Let $(A,B)\in\Sigma_\text{pk}$. Then, there exists a $k\in\{1,\ldots,q\}$ such that
\begin{equation}\label{eq:Lyap_P^2 - app}
\alpha^2 P_{k}-(A+BK_{k})^\top P_{k}(A+BK_{k})\geq I.
\end{equation}
This implies that if $z_2(T)=k$, then for all $t\geq T$ we have $x(t)^{\top} P_{z_2(t)} x(t)\leq z_1(t)$. Therefore, the set $\mathcal{T}$ is either empty or finite with at most $q-1$ elements. 
Let $0\leq\bar{q}\leq q-1$ denote the number of elements of $\mathcal{T}$.
Define $T_0=0$ and $T_{\bar{q}+1} = \infty$.
If $\mathcal{T}\neq\varnothing$, let $T_i$, $i\in\{1,\ldots,\bar{q}\}$, denote the elements of $\mathcal{T}$ ordered as $T_1\leq T_2\leq\cdots\leq T_{\bar{q}}$. 

Suppose $t\in [T_i,T_{i+1})\cap \mathbb{Z}$ for some $i\in\{0,\ldots,\bar{q}\}$. 
In case $t=T_0$, then $\|x(t,x_0,z_0)\|\leq c\eta^t\|x_0\|$ is trivially satisfied for any $c\geq 1$ and $0\leq\eta< \alpha$.
In case $t=T_i$ for some $i\in\{1,\ldots,\bar{q}\}$, the identity \mbox{$x(T_i)=(A+BK_{z_2(T_i)})x(T_i-1)$} implies that 
\begin{equation}
\label{eq:state3}
\norm{x(T_i)}\leq \norm{A+BK_{z_2(T_i)}}\norm{x(T_i-1)}\leq \kappa \norm{x(T_i-1)}.
\end{equation}
In case $t\in (T_i,T_{i+1})\cap\mathbb{Z}$ for some $i\in\{0,\ldots,\bar{q}\}$,  we have $z_2(t)=z_2(T_i)$, $x(t)^{\top} P_{z_2(t)} x(t)\leq z_1(t)$, and \mbox{$z_1(t)=\alpha^2 x(t-1)^\top P_{z_2(t)} x(t-1)-\|x(t-1)\|^2$}. Thus, we have
\begin{equation} 
x(t)^\top P_{z_2(T_i)}x(t)\leq \alpha^2 x(t-1)^\top P_{z_2(T_i)}x(t-1)-\|x(t-1)\|^2.
\end{equation}
This, together with the inequality $x(t-1)^\top P_{z_2(T_i)}x(t-1)\leq \lambda_\text{max}\|x(t-1)\|^2$,
implies that
\begin{equation}
x(t)^\top P_{z_2(T_i)}x(t)\leq (\alpha^2-\tfrac{1}{\lambda_\text{max}}) x(t-1)^\top P_{z_2(T_i)}x(t-1).
\end{equation}
Hence, we have
\begin{equation}
\label{eq:lyap_first_case}
x(t)^\top P_{z_2(T_i)}x(t)\leq (\alpha^2-\tfrac{1}{\lambda_\text{max}})^{t-T_i} x(T_i)^\top P_{z_2(T_i)}x(T_i).
\end{equation}
Therefore, using $x(T_i)^\top P_{z_2(T_i)}x(T_i)\leq\lambda_\text{max}\|x(T_i)\|^2$ and $x(t)^\top P_{z_2(T_i)}x(t)\geq\lambda_\text{min}\|x(t)\|^2$, we have
\begin{equation}
\label{eq:state0}
\|x(t)\|\leq \sqrt{\tfrac{\lambda_\text{max}}{\lambda_\text{min}}}(\alpha^2-\tfrac{1}{\lambda_\text{max}})^{(t-T_i)/2} \|x(T_i)\|.
\end{equation} 
By repeatedly applying \eqref{eq:state3} and \eqref{eq:state0}, one can verify that 
\begin{equation}
    \norm{x(t)} \leq \kappa^{i}\left(\tfrac{\lambda_\text{max}}{\lambda_\text{min}}\right)^{(i+1)/2}(\alpha^2-\tfrac{1}{\lambda_\text{max}})^{t/2} \|x(T_0)\|.
\end{equation}
Define $\bar{\kappa}\coloneqq \max\{1,\kappa^{q-1}\}$.
Then, since $i\leq q-1$, this shows that 
\begin{equation}\label{eqn: bound on norm of x}
    \| x(t)\| \leq \bar{\kappa} \left(\tfrac{\lambda_\text{max}}{\lambda_\text{min}}\right)^{q/2} (\alpha^2-\tfrac{1}{\lambda_\text{max}})^{t/2} \|x(T_0)\|
\end{equation} 
for all $t\in\mathbb{Z}_+$. 
This proves that property \ref{(P4)} holds for $\alpha\leq 1$, 
\mbox{$\eta = \sqrt{\alpha^2-\tfrac{1}{\lambda_\text{max}}}$}, and $c = \bar{\kappa} \left(\tfrac{\lambda_\text{max}}{\lambda_\text{min}}\right)^{q/2}\geq 1$.
We note that \mbox{$0\leq\eta<\alpha$} since  \eqref{eq:Lyap_P^2 - app} implies that $\lambda_\text{max}\geq\tfrac{1}{\alpha^2}$.

Next, we show that property~\ref{(P1)} holds. Let $(A,B)\in\Sigma_\text{pk}$, \mbox{$x_0\in\mathbb{R}^n$}, and $z_0\in\mathbb{R}^2$. 
Denote $z_0=\begin{bmatrix} z_{10} & z_{20}\end{bmatrix}^\top$. Take $\gamma$ as
\begin{equation}
\gamma(x_0,z_0) \coloneqq \max\{|z_{10}|,\xi(x_0)\} + \max\{|z_{20}|,q\},
\end{equation}
where $\xi(x_0) \coloneqq  c^2(\alpha^2\lambda_\text{max}-1) \norm{x_0}^2$. Observe from the triangle inequality and the partitioning of the controller state that \mbox{$\norm{z(t)} \leq |z_1(t)| + |z_2(t)|$} holds for all $t\in\mathbb{Z}_+$. This implies that $\norm{z(0)} \leq \gamma(x_0,z_0)$.  
For $t\geq 1$, we use 
\begin{equation}
\label{eq:last}
z_1(t) = \alpha^2 x(t-1)^\top P_{z_2(t)} x(t-1) - \|x(t-1)\|^2 
\end{equation}
to verify that $|z_1(t)|  \leq  (\alpha^2\lambda_\text{max}-1)\|x(t-1)\|^2$. Moreover, from \eqref{eqn: bound on norm of x}, we have that \mbox{$\|x(t)\|\leq c\|x_0\|$} for all $t$. This, together with \eqref{eq:last} and $\alpha\leq 1$, implies that $\norm{z(t)} \leq \gamma(x_0,z_0)$. Therefore, property~\ref{(P1)} holds. \hfill \QED